\documentclass[12pt]{amsart}
\usepackage{amssymb}
\usepackage{amsmath, amscd}
\usepackage{amsthm}
\usepackage{xcolor}
\usepackage{ulem}

\newtheorem{theorem}{Theorem}[section]

\newtheorem{proposition}[theorem]{Proposition}
\newtheorem{lemma}[theorem]{Lemma}

\newtheorem{corollary}[theorem]{Corollary}
\theoremstyle{definition}

\newtheorem{example}[theorem]{Example}
\newtheorem{definition}[theorem]{Definition}

\newtheorem{conjecture}[theorem]{Conjecture}

\def\val#1{\vert #1 \vert}

\usepackage[utf8]{inputenc}
\usepackage{amssymb}
\usepackage{amsmath, amscd}
\usepackage{amsthm}
\usepackage{xcolor}

\begin{document}

\author[P.V. Danchev]{Peter V. Danchev}
\address{Institute of Mathematics and Informatics, Bulgarian Academy of Sciences, 1113 Sofia, Bulgaria}
\email{danchev@math.bas.bg; pvdanchev@yahoo.com}
\author[P.W. Keef]{Patrick W. Keef}
\address{Department of Mathematics, Whitman College, Walla Walla, WA 99362, USA}
\email{keef@whitman.edu}

\title[Mixed Abelian Groups With Bounded $p$-Torsion] {Generalized Bassian and Other Mixed \\ Abelian Groups With Bounded $p$-Torsion}

\maketitle

\begin{abstract} It is known that a mixed abelian group $G$ with torsion $T$ is Bassian if, and only if, it has finite torsion-free rank and has finite $p$-torsion (i.e., each $T_p$ is finite).  It is also known that if $G$ is generalized Bassian, then each $pT_p$ is finite, so that $G$ has bounded $p$-torsion. To further describe the generalized Bassian groups, we start by characterizing the groups in some important classes of mixed groups with bounded $p$-torsion (e.g., the balanced-projective groups and the Warfield groups). We then prove that all generalized Bassian groups must have finite torsion-free rank, thus answering a question recently posed in Acta Math. Hung. (2022) by Chekhlov-Danchev-Goldsmith. This implies that every generalized Bassian group must be a B+E-group; i.e., the direct sum of a Bassian group and an elementary group. The converse is shown to hold for a large class of mixed groups, including the Warfield groups. It is also proved that $G$ is a B+E-group if, and only if, it is a subgroup of a generalized Bassian group.
\end{abstract}

\vskip2.0pc

\section{Introduction}

All groups considered will be abelian and additively written. Our notation and terminology will mainly agree with the classical book \cite{F}. For instance, if $G$ is a group, the letter $T=T(G)$ will denote its maximal torsion subgroup with $p$-primary component $T_p=T_p(G)$.
If  $x\in G$, then $|x|_p$ will denote its $p$-height.  A torsion group $T$  is {\it elementary} if each $T_p$ is $p$-bounded. Clearly, any subgroup of an elementary group is a summand.

\medskip

The following notions are central to our investigations.

\begin{definition} (\cite{CDG1},\cite{CDG2}) A group $G$ is said to be {\it Bassian} if it cannot be embedded in a proper homomorphic image of itself or, equivalently, if the existence of an injection $G \to G/N$, for some subgroup $N$ of $G$, forces that $N = \{0\}$.

More generally, if this injective homomorphism implies that $N$ is a direct summand of $G$, then $G$ is called {\it generalized Bassian}.
\end{definition}

Restated slightly, the Main Theorem of \cite{CDG1} is that $G$ is Bassian if, and only if, it has finite torsion-free rank and, for all primes $p$, $T_p$ is also finite. The generalized Bassian groups were not fully characterized in \cite{CDG2}, but it was shown that if $G$ is generalized Bassian, then $pT_p$ must be finite for all primes $p$ (\cite{CDG2}, Lemma~3.1). The main goal of this paper is to describe the class of generalized Bassian groups in greater detail.  In pursuing this goal we will discuss a number of results applicable to any mixed group with bounded $p$-torsion (i.e., all $T_p$ are bounded). Concretely, our work is organized thus:

\medskip

In Section~\ref{A} we introduce some examples of groups with bounded $p$-torsion.
Two of the most important classes of mixed abelian groups are the {\it Warfield groups}, and its subclass consisting of the {\it balanced-projective groups}, introduced in \cite{W} and \cite{HR}, respectively (we define these classes later; see also \cite{HW} and \cite{L}).   We completely characterize the groups in these classes which have bounded $p$-torsion  (Propositions~\ref{Warfield} and~\ref{balanced}).

\medskip

In Section~\ref{B} we prove that, as with Bassian groups, any generalized Bassian group necessarily has finite torsion-free rank (Theorem~\ref{finite}). This gives an affirmative answer to a question from \cite {CDG2} (Section~4). Second, we show that any group $G$ with finite torsion-free rank and bounded $p$-torsion must necessarily be isomorphic to a direct sum $H\oplus S$, where $H$ is a Bassian group and $S$ is torsion (Theorem~\ref{decompose}). This gives much more direct connection between the classes of Bassian and generalized Bassian than was apparent in \cite{CDG2}.

The following definition is important:

\begin{definition} A group $G$ is called {\it B+E} (or alternatively, is said to be a {\it B+E-group}) if it decomposes as $G=A\oplus E$, where $A$ is Bassian and $E$ is elementary.
\end{definition}

We show that $G$ is B+E if, and only if, it has finite torsion-free rank and for all primes $p$, $pT_p$ is finite (Proposition~\ref{step}), which implies that any subgroup of a B+E-group shares that property (Corollary~\ref{subgroup}). It also follows that every generalized Bassian group is B+E (Corollary~\ref{oneway}). The converse of this implication is the focus of the remainder of the paper.

\begin{conjecture}\label{conj}
A group is \rm{B+E} if, and only if, it is generalized Bassian; i.e., the converse of Corollary~\ref{oneway} also holds.
\end{conjecture}

We prove that the class of B+E-groups for which the conjecture holds is quite large and includes all those mentioned in Section~\ref{A}.
An important step in this investigation is the observation that if $G=A\oplus E$ is B+E with $E$ elementary and $A$ has no infinite elementary summands, then $G$ is generalized Bassian (Theorem~\ref{allfinite}).

\medskip

In Section~\ref{C} we prove that many B+E-groups have a decomposition as that in Theorem~\ref{allfinite}. This includes the Warfield groups (and hence the balanced-projectives). Even more generally, we show that if $G$ is a B+E-group with a nice free subgroup $F$ such that $G/F$ is torsion (i.e., $F$ is {\it quasi-essential} in $G$), then $G$ is generalized Bassian (Theorem~\ref{nice2}). This, in turn, allows us to conclude that $G$ is a B+E-group if, and only if, it is a subgroup of a generalized Bassian group (Corollary~\ref{embeds}).

\section{Groups with bounded $p$-torsion}\label{A}

We begin by reviewing some terminology. Suppose $G$ is a group with torsion subgroup $T$. Recall, $G$ has {\it bounded $p$-torsion} if $T_p$ is bounded for all primes $p$. So, a generalized Bassian group (and hence a Bassian group) has bounded $p$-torsion; in fact, each $pT_p$ (respectively, $T_p$) must be finite.

We will use some basic ideas and terminology from the theory of valuated groups (see, for example, \cite{RW}). Our valuated groups will have a valuation for every prime $p$, which generalizes the $p$-height on a group.

For a valuated group $V$, prime $p$ and ordinal $\alpha$, let $V(\alpha)=\{x\in V: \val x_p\geq \alpha\}$ and
$$U_V^p(\alpha)=\{x\in V(\alpha):\val {px}_p\geq \alpha+2\}/V(\alpha+1)$$
be the $\alpha$th Ulm invariant of $V$ and $f_V^p(\alpha)$ be the corresponding value of the $p$-Ulm function of $V$.

If $V$ is a valuated group, then by a {\it realization of $V$} we mean a group $G$ containing $V$ as a subgroup such that the $p$-valuations on $V$ agree with the $p$-heights inherited from $G$. Note that, in this case, for every prime $p$ and ordinal $\alpha$, there is a natural injection $U_V^p(\alpha)\to U_G^p(\alpha)$ whose co-kernels can be identified with the relative Ulm invariants of $V$ in $G$. We will primarily be concerned with {\it nice} realizations, i.e., where $V$ is also a nice subgroup of $G$.

If $G$ is a group, or $V$ is a valuated group, and $p$ is a prime, we will denote their localizations by $G_p$ and $V_p$, respectively; so, $G_p=G\otimes \mathbb Z_{(p)}$ and $V_p=V\otimes \mathbb Z_{(p)}$. In particular, this is consistent with the notation $T_p$ for the $p$-torsion subgroup of $G$. There is a $p$-height preserving natural map $G\to G_p$ whose kernel is $\oplus_{q\ne p} T_q$.

Suppose $G$ is a realization of the free valuated quasi-essential subgroup $F$.  We will say that $F$ is {\it tight}  if for all primes $p$, $U^p_F(j)\to U^p_G(j)$ is an isomorphism for all $j<\omega$. Similarly, we will say $F$ is {\it 0-tight} if for all primes $p$, $U^p_F(0)\to U^p_G(0)$ is an isomorphism.

A quasi-essential free subgroup $F$ of a group $G$ will be called a {\it decomposition subgroup} if it is generated by a decomposition basis for $G$. That is, $F$ is nice in $G$ and it is the valuated direct sum of a collection of cyclic summands. So $G$ is a {\it Warfield} group (i.e., a summand of a simply presented group) exactly when it has a decomposition subgroup such that the factor-group $G/F$ is totally projective (where we are allowing divisible torsion groups to be included as totally projective).

Note that if $G$ has bounded $p$-torsion, then $T_p$ must, as a pure bounded subgroup of $G$, be a direct summand; i.e., $G=T_p\oplus A_p$ for some $A_p$. If $x\in T_p$, then clearly $\val x$ is finite or $\infty$. Similarly, if $x\in A_p$, then since multiplication by $p$ is an injective homomorphism $A_p\to A_p$, it still follows that $\val x$ is finite or $\infty$. Therefore, the same holds for any $x\in G$.

We will apply the following (essentially well-known) lemma to the case of Warfield groups. Later, we prove a similar result for a broader class of valuated groups (see Lemma~\ref{tight} below).

\begin{lemma}\label{realize}
Any infinite cyclic valuated group $X=\langle x\rangle$ such that, for all primes $p$ and $k<\omega$, $\val {p^k x}_p$ is either finite or $\infty$, has a tight realization $H$ that is simply presented of torsion-free rank $1$.
\end{lemma}

\begin{proof} For each prime $p$, consider the set $K_p\subseteq \omega$ of all $k$ such that either $k=0$ or $\val {p^k x}\ne \val {p^{k-1}x}_p+1$ (where $\infty=\infty+1$). For each prime $p$ and $k\in K_p$, define a set of generators, $Y_p^k$, and relations as follows: If $\val {p^k x}=\infty$, let $Y_p^k=\{y_{p,j}^k\}_{j<\omega}$ with relations $y_{p,0}^k=p^k x$ and $py_{p,j+1}^k =y_{p,j}^k$. On the other hand, if $0\ne \val {p^k x}=n<\omega$, let  $Y_p^k=\{y_{p, j}^k\}_{j\leq n}$ with relation $y_{p,0}^k=p^k x$ and $py_{p,j+1}^k =y_{p,j}^k$

If $H$ is the group generated $x$ and the union of all the $Y_p^k$ with the above relations, then $H$ is clearly simply presented, and a standard argument shows that it is a realization of $X$. To show that $X$ is tight in $H$, suppose $p$ is some fixed prime. It is fairly straightforward (though somewhat tedious) to observe that if $T$ is the torsion subgroup of $H$, then $T[p]$ is a free valuated vector space with a basis generated by the images of all $p^{k-1} x - y_{p,1}^k$ such that $0<k\in K_p$, and that, similarly, the non-zero Ulm invariants of $X$ are generated by the images of all such $p^{k-1} x$, showing that $U_X^p(j)\to U_H^p(j)$ is always an isomorphism, as promised.
\end{proof}

Though it could have been verified directly, in the last result, since $H$ has torsion-free rank $1$, the subgroup $X$ must be nice in $H$.

\begin{proposition}\label{flag}
Suppose $G$ is a Warfield group with bounded $p$-torsion. If $F$ is a decomposition subgroup of $G$, then $G=H\oplus S$, where $H$ is a nice tight simply presented realization of $F$ and $S$ is torsion.
\end{proposition}

\begin{proof}  Let $\{x_i\}_{i\in I}$ be a decomposition basis for $F$. For each $i\in I$, $X_i:=\langle x_i\rangle$ clearly satisfies the hypotheses of Lemma~\ref{realize}; let $H_i$
be a rank 1 simply presented tight realization of $X_i$ and $H$ be the (external) direct sum of the $H_i$. Construct a direct sum of cyclics $S$ such that for all primes $p$ and $j<\omega$, we have $f^p_F(j)+f^p_S(j)=f^p_G(j)$. It follows that the relative Ulm invariants of $F$ in $G$ agree with those of $F$ in $A\oplus S$ (both agree with $f_S^p(j)$ for all $p$ and $j$). Therefore, the identity map $F\to F$ extends to an isomorphism $G\to H\oplus S$, as required.
\end{proof}

The following describes all Warfield groups with bounded $p$-torsion.

\begin{proposition}\label{Warfield}
Suppose $G$ is a group with bounded $p$-torsion. Then the following are equivalent:

(a) $G$ is simply presented;

(b) $G$ is the direct sum of a set of groups of torsion-free rank at most 1;

(c)  $G$ is a Warfield group.
\end{proposition}

\begin{proof} That (a) implies (b) is well-known.

To show (b) implies (c), it is clear that if $G$ is torsion, then it is a direct sum of cyclic groups and hence simply presented. So, assume $G$ has torsion-free rank~1. Let $x\in G$ have infinite order. If $X=\langle x\rangle$ then it is well known that $X$ is nice in $G$. For each prime $p$, we claim that $(G/X)_p$ is isomorphic to $B\oplus D$, were $B$ is bounded and $D$ is either $0$ or $\mathbb Z_{p^\infty}$. This will imply that $G/X$ will be totally projective, so that $G$ is Warfield.

To verify the claim, let $p$ be a prime and consider the localizations $X_{p}$ and $G_{p}$; so $G_{p}/X_p$ is isomorphic to $(G/X)_p$. Note that $T_{p}$ is a summand of $G_{p}$; suppose $G_{p}=T_{p}\oplus Y_{p}$. Since $Y_{p}$ is either isomorphic to $\mathbb Z_{(p)}$ or $\mathbb Q$, it is readily checked that in the first case $(G/X)_p$ is bounded and in the second case, it is isomorphic to the direct sum of a bounded group and a copy of $\mathbb Z_{p^\infty}$.

Finally, (c) implies (a) by Proposition~\ref{flag}.
\end{proof}

The following standard criterion is a direct consequence of Proposition~\ref{Warfield}.

\begin{corollary}\label{free}
A torsion-free group is Warfield if, and only if, it is completely decomposable.
\end{corollary}

We now describe the balanced-projective groups with bounded $p$-torsion.
Imitating \cite{W}, recall that a {\it generalized height} $h$ means the formal direct product $\prod_{p\in \mathbb{P}} p^{v(p)}$ (where $\mathbb P$ is the set of primes), and if $G$ is a group, then $h(G):=\bigcap_{p\in \mathbb{P}} p^{v(p)}G$, where $v(p)$ is either an ordinal or the symbol $\infty$. A short-exact sequence $$0\to A\to B\to C\to 0$$ is termed {\it balanced} if, for every generalized height $h$, the induced sequence $$0\to h(A)\to h(B)\to h(C)\to 0$$ is exact. A group $P$ is said to be {\it balanced-projective} if it is projective with respect to all such sequences, that is, for any such sequence, the induced map $\mathrm{Hom}(P,B)\to \mathrm{Hom}(P,C)$ is surjective.

The following assertion describes all balanced-projective groups with bounded $p$-torsion.

\begin{proposition}\label{balanced}
Suppose $G$ is a group with bounded $p$-torsion. Then the following are equivalent:

(a) $G$ is balanced-projective;

(b) $G$ is a splitting mixed Warfield group;

(c) $G\cong T\oplus A$, where $A$ is a completely decomposable torsion-free group.
\end{proposition}

\begin{proof} Suppose that (a) holds, i.e., $G$ is balanced-projective. It is well known that $G$ must be a Warfield group. Note that, for every torsion-free element $x\in G$, the coordinates of $h(x)$ are integers or $\infty$. It follows that there is a torsion-free group of rank $1$, say $Q_x$, with $a\in Q_x$ and $h(a)=h(x)$. The assignment $a\mapsto x$ defines a homomorphism $Q_x\to G$. These determine a homomorphism $$\overline G:=T\oplus(\oplus_{x\in {G\setminus T}} Q_x)\to G,$$ which is clearly surjective and balanced-exact. Thus, it follows that $G$ is isomorphic to a summand of $\overline G$ containing $T$, so that $G$ is necessarily a splitting mixed group, as stated.

That (b) implies (c) follows directly from Corollary~\ref{free}. Finally, since (torsion) cyclic groups and rank-1 torsion-free groups are clearly balanced-projective, (c) implies (a), completing the proof.
\end{proof}

The following example is crucial.

\begin{example}
(a) There is a Bassian group that is {\it not} Warfield.

(b) There is a Bassian Warfield group that is {\it not} balanced-projective.
\end{example}

\begin{proof}
For (a), we just consider any torsion-free group of finite rank that is not completely decomposable. For example, any pure subgroup of the $p$-adic integers, $J_p$, of finite rank exceeding~1 will satisfy the requirements. For (b), let $G$ be generated by $t$ together with $s_p$ for each prime $p$, subject to the relations $p^2 s_p =pt$. Clearly, $G$ is simply presented, and hence Warfield. Since $G$ has rank~1 and $T_p\cong \mathbb Z_p$ for each prime, $G$ is Bassian. However, it is easy to check that $[t]\in G/T$ has $h([t])=(1,1,1,\dots)$ where as, for any torsion-free element $x\in G$, $\val x_p=0$ for all but finitely many primes $p$. It follows that $T$ is not a summand of $G$, so that $G$ is not balanced-projective, as wanted.
\end{proof}

Suppose $T$ is a torsion group with bounded $p$-torsion. Note that $\prod_p T_p/T$ is always torsion-free divisible. Suppose $G$ is a pure subgroup of $\prod_p T_p$ containing $T$. For want of a better term, we will call such a group {\it PSP} or just a {\it PSP-group} for ``pure subgroup of the product". Alternately, if $G$ has bounded $p$-torsion, then $G$ is PSP if, and only if, $G$ is Hausdorff in the $\mathbb Z$-adic topology and $T$ is dense in $G$ in that topology. This follows since $\prod_p T_p$ will be the completion of $T$ in the $\mathbb Z$-adic topology.

\begin{proposition}\label{PSP} A {\rm{PSP}}-group $G$ is Warfield if, and only if, there are decompositions $T=\oplus_{i\in I}T^i$ and $G=\oplus_{i\in I}G^i$ such that, for each $i\in I$, $G^i$ is a PSP-group of rank $1$ and torsion subgroup $T^i$.
\end{proposition}

\begin{proof}
Certainly, if $G$ is Warfield, then by Theorem~\ref{Warfield}, $G$ must decompose into groups of torsion-free rank $1$, as indicated. Clearly, each $G_i$ will be Hausdorff in the $\mathbb Z$-adic topology with $T_i$ as a dense subgroup. Therefore, each $G_i$ will also be a PSP-group.

The converse is an immediate consequence of Theorem~\ref{Warfield}.
\end{proof}

\begin{example}
There is a Bassian PSP-group of rank $2$ that is not Warfield.
\end{example}

\begin{proof}
Let $(a_j,b_j)$ for $j\in \mathbb N$ be an enumeration of $\mathbb Z^2\setminus \{(0,0)\}$. For each $j\in \mathbb N$, we can clearly find $x_j, y_j\in \mathbb Z$ and a prime $p_j>p_{j-1}$ such $p_j>\vert a_j x_j+b_jy_j\vert\ne 0$.

Let $T=\oplus_{j\in \mathbb N} \mathbb Z_{p_j}$ and  $x:=[x_j]$, $y:=[y_j]$ in  $P:=\prod_{j\in \mathbb N} \mathbb Z_{p_j}$. If $a_j$ and $b_j$ are integers, not both 0, and we consider the corresponding linear combination $a_j x+b_j y\in P$, then the entire point of our construction is that such a sum can be 0 for only a finite number of coordinates (those for $i<j$). In particular, $x$ and $y$ are linearly independent over $T$ and there is a PSP- group of torsion-free rank $2$ containing both.

If there were decompositions $T=T^1\oplus T^2$ and $G=G^1\oplus G^2$ as in Proposition~\ref{PSP}, then clearly both $T^1$ and $T^2$ must have an infinite number of non-zero components. Therefore, no linear combination of $x$ and $y$ can be a non-zero element of either summand, which is clearly a contradiction.

Therefore, $G$ is a Bassian PSP-group, but not Warfield, as claimed.
\end{proof}

\section{Groups with bounded $p$-torsion and finite torsion-free rank }\label{B}

We begin by characterizing the groups described in the title to this section.

\begin{theorem}\label{decompose} The group $G$ has finite torsion-free rank and bounded $p$-torsion if, and only if, is decomposable as $G=H\oplus S$, where $H$ is Bassian and $S$ is a torsion group with bounded $p$-torsion.
\end{theorem}

\begin{proof} If $G=H\oplus S$ as above, it is easy to check that $r_0(G)=r_0(H)$ is finite and $G$ has  bounded $p$-torsion. So, assume $G$ satisfies these two conditions.

Let $D$ be a divisible hull of $G/T$, so $D$ is a direct sum of a finite number of copies of $\mathbb Q$. Since the map $${\rm Ext}(D, T)\to {\rm Ext} (G/T, T)$$ is surjective, there is a group $G'$ containing $G$ such that $G'/T\cong D$. If there is a similar decomposition $G'=H'\oplus S'$ for $G'$, it readily follows that $G'$ also has torsion $T$ and we can simply let $S=S'$ and $H=H'\cap G$. So, there is no loss of generality in assuming that $G/T= D$ is divisible.

Let $P=\prod_{p\in \mathbb P} T_p$; so we can think of $T$ as a subgroup of $P$ and $Q:=P/T$ is torsion-free divisible. Now, the short-exact sequence $$0\to T\to P\to Q\to 0$$ determines a long-exact sequence, which determines an isomorphism, as follows:
$$
[{\rm Hom}(D, P)=]\ \ 0 \to {\rm Hom}(D, Q) \to {\rm Ext}(D, T) \to {\rm Ext}(D, P)\to \cdots
$$
Note that $${\rm Ext}(D, P)= {\rm Pext}(D, P)\cong \prod_{p\in \mathbb P} {\rm Pext}(D, T_p)=0,$$ since $T_p$ is bounded, and hence pure-injective.  Therefore, ${\rm Hom}(D, Q) \to {\rm Ext}(D, T)$ is an isomorphism.

We can think of the extension $$\xi: 0 \to T \to G \to D \to 0$$ as an element of ${\rm Ext}(D, T)$; in fact, suppose $\phi:D \to Q$ is the corresponding element of ${\rm Hom}(D, Q).$  Let $y_1, \dots, y_n$ be a maximal linearly independent subset of $D$. So, $\phi$ is completely determined by the $n$ elements $\phi(y_1),\dots, \phi(y_n)$ of $Q$. For $j=1,\dots, n$, choose $x_j\in P$ so that $\phi(y_j)= x_j+T$. Note that for each prime $p$, the $T_p$ coordinates from all of the $x_j$ can be chosen to come from a single finite summand of $T_p$. In other words, for each prime $p$ there is a decomposition $T_p= F_p\oplus S_p$, where $F_p$ is finite, such that if $F=\oplus F_p$ and $P_F=\prod F_p$, then $x_j\in P_F$ for each $j$.

So, if we also set
 $S=\oplus S_p$ and $P_S=\prod S_p$, then in the obvious isomorphism
$$
            Q=P/T \cong  P_F/F  \oplus P_S/S,
$$
each $\phi(y_j)\in P_F/F$ for $j=1, \dots, n$.

It follows that in the corresponding decomposition $${\rm Hom}(D, Q)\cong {\rm Hom}(D, P_F/F)\oplus {\rm Hom}(D, P_S/S)$$ that we can view $\phi$ is an element of the first summand. Therefore, in the other corresponding decomposition $${\rm Ext}(D, T)\cong {\rm Ext}(D, F)\oplus {\rm Ext}(D, S),$$ we can think of $\xi$ also as being an element of the first summand. This immediately implies the required decomposition, as asserted.
\end{proof}

We continue with the following perhaps unsurprising result, a case of which will be used in the sequel.

\begin{proposition}\label{quotient} Suppose $G$ is a group of finite torsion-free rank with bounded $p$-torsion. If $K$ is any subgroup of $G$ such that $G/K$ is torsion, then $G/K$ is actually totally projective. In fact, for each prime $p$, $(G/K)_p\cong G_p/K_p$ will be the direct sum of a divisible group and a bounded group.
\end{proposition}

\begin{proof} Let $\overline G=G/K$ and $\pi: G\to \overline G$ be the canonical epimorphism. There is clearly no loss of generality in assuming that $G$ and $K$ are $p$-local groups. It follows that there is a direct sum decomposition $G=T\oplus F$, where $F$ is torsion-free (and $p$-local). It readily follows from  the theory of finite rank torsion-free $p$-local groups that $\pi(F)\cong F/[F\cap K]$ is the direct sum of a divisible module and a bounded module (this follows since $F/(A\cap F)$ must have finite $p$-rank).  It is also obvious that $\pi(T)$ is bounded. An elementary argument then shows that $\overline G=\pi(F)+\pi(T)$ is also the direct sum of a divisible module and a bounded module.
\end{proof}

\medskip

We now restrict our attention to a particular subclass of the groups described in Theorem~\ref{decompose}. Recall once again that the group $G$ is B+E if $G=B\oplus E$, where $B$ is Bassian and $E$ is elementary. The following gives an alternate characterization of these groups.

\begin{proposition}\label{step}
A group $G$ is {\rm{B+E}} if, and only if, it has finite torsion-free rank and $pT_p$ is finite for every prime $p$.
\end{proposition}

\begin{proof} Certainly, the above two conditions are necessary for $G$ to be a B+E-group, so suppose $r_0(G)$ is finite and $pT_p$ is finite for all prime $p$. It follows that the hypotheses of Theorem~\ref{decompose} hold; so suppose $G=H\oplus S$ as in that result. Note that each $S_p$ has only a finite number of summands of the form $\mathbb Z_{p^k}$ for $k>1$. Transferring each of these from $S$ to $H$ gives the required decomposition into a Bassian group and an elementary group.
\end{proof}

As a direct consequence, we manifestly obtain that the class of B+E-groups is closed under taking subgroups.

\begin{corollary}\label{subgroup}
Every subgroup of a {\rm{B+E}}-group is also a {\rm{B+E}}-group.
\end{corollary}

\begin{proof}
The two conditions in Proposition~\ref{step} are clearly inherited by all subgroups.
\end{proof}

Again, the difference between the B+E-groups and those described in Theorem~\ref{decompose} is that, in the later result, we have that each $T_p$ is bounded, but if $G$ is B+E, we demand that each $pT_p$ is finite.

We know that if $G$ is generalized Bassian, then it has bounded $p$-torsion; in fact, for every prime $p$, $pT_p$ must be finite. The following statement, which is a positive answer to a question from \cite {CDG2}, shows that such a group must also have finite torsion-free rank.

\begin{theorem}\label{finite}
A generalized Bassian group $G$ has finite torsion-free rank.
\end{theorem}

\begin{proof}
Suppose $\kappa:= r_0(G)$ is infinite. We can clearly find a subgroup $C\subseteq G$ such that $|C|= \kappa$ and $G=C+T$.  We know for each prime $p$ there is a decomposition $T_p=R_p\oplus S_p$, where $R_p$ is finite and $S_p$ is elementary. Let $A=C+({\oplus R_p})$. Note that $A$ has cardinality $\kappa$ and
$$
             {\oplus R_p}\subseteq (T\cap A)\subseteq T= ({\oplus R_p})\oplus ({\oplus S_p}).
$$
It follows that we can find $E\subseteq {\oplus S_p}$ such that $T=(T\cap A)\oplus E$. Therefore, $A\cap E\subseteq (T\cap A)\cap E=0$ and  $G=C+T\subseteq A+T=A+(T\cap A)+ E=A+E\subseteq G$; so we must have $G=A\oplus E$.

Let $F$ be a free subgroup of $A$ with $A/F$ is torsion. If $D$ is a divisible hull of $A$, then it easily follows that $r_0(F)=r_0(D)=|D|=\kappa $ and there is a surjective homomorphism $F\to D$, which we can extend to a surjective homomorphism $A\to D$. If $K$ is the kernel of this homomorphism, then $G/K\cong D\oplus E$. Since there is an obvious embedding $$G=A\oplus E\to D\oplus E\cong G/K$$ and the group $G$ is generalized Bassian, there must be a decomposition
$$
G\cong D\oplus E\oplus K.
$$
This obviously implies that $G$ has a summand that is a torsion-free divisible group of infinite rank, which cannot be the case (since if $X$ is a cyclic subgroup of such a summand $Z$, then $Z$ embeds in $Z/X$, so that $G$ clearly embeds in $G/X$; but $X$ is not a summand of $G$).
\end{proof}

The next consequence immediately follows from a combination of Theorem~\ref{finite} and Proposition~\ref{step}.

\begin{corollary}\label{oneway}
Every generalized Bassian group is \rm{B+E}.
\end{corollary}

We now pause for a couple of observations about summands. Suppose $G=A\oplus B$ and $\gamma:B\to A$ is any homomorphism. If we set $$B'=\{(\gamma(x),x): x\in B\},$$ then it is well known that $G=A\oplus B'$; so, in particular, that $B'$ is another summand of $G$. In addition, if $\pi:G=A\oplus B\to B$ is the obvious projection, then suppose $M$ is a subgroup of $G$ such that $\pi$ restricted to $M$ is injective and there is a decomposition $B=\pi(M)\oplus K$. Then clearly, $M\cap K=0$ and if $B'=M\oplus K$, then there is a decomposition $G=A\oplus B'$. Actually, this second observation follows from the first: Let $\tau: G\to A$ be the other projection and define $\gamma:B=\pi(M)\oplus K\to A$ to be 0 on $K$ and $\tau\circ \pi^{-1}$ on $\pi(M)$.

\medskip

Preparing for the proof of Theorem~\ref{allfinite}, suppose $G$ is a B+E-group,  $N$ is a subgroup of $G$ and there is an injection $G\to G/N$. Our injection shows $r_0(G)\leq r_0(G/N)\le r_0(G)$, which implies that $N$ has no elements of infinite order. So, $N\subseteq T$ and $T/N$ is the torsion subgroup of $G/N$.

For each prime $p$,
$$p(T/N)_p\cong [pT_p+N_p]/N_p\cong pT_p/[pT_p\cap N_p]$$ has at most as many elements as $pT_p$, which is finite. Therefore, the derived injection $pT_p\to p(T/N)_p$ must be an isomorphism. Thus, $N_p\cap pT_p=0$, i.e., $N_p$ is an elementary summand of $T_p$. In addition, the image of $T_p\to G/N$ must be a pure in $(T/N)_p$, and hence a summand. So, we have proved the following technicality.

\begin{lemma}\label{Nelem}
Suppose $G$ is a {\rm {B+E}}-group, $N$ is a subgroup of $G$ and $G\to G/N$ is an injective homomorphism. Then $N$ is an elementary group that is a summand of $T$. Further, if $\hat T$ is the image of $T$ under this injection, then $\hat T$ is a summand of $T/N$ such that, for all primes $p$, $p\hat T_p=p(T/N)_p$.
\end{lemma}

The following result is key to showing that the B+E-groups in many classes satisfy Conjecture~\ref{conj}.

\begin{theorem}\label{allfinite}
Suppose $G=A\oplus E$ is a {\rm{B+E}}-group, where $E$ is elementary and every elementary summand of $A$ is finite. Then $G$ is generalized Bassian.
\end{theorem}

\begin{proof} We observe first that $A$ must be Bassian: Certainly, since $G$ has finite torsion-free rank, so does $A$. Let $p$ be a prime; we need to show $T_{A,p}$ is finite. Since $A$ has no infinite elementary summands, we can conclude that $f_A^p(0)$ is finite. And since $p(T_{A,p})\subseteq  p(T_{G,p})$, which is finite, we can conclude that $p(T_{A,p})$ is finite. These facts readily imply that $T_{A,p}$ is finite, as required.

Suppose $N$ is a subgroup of $G$ such that there is an embedding $G\to G/N$; let $\hat G\subseteq G/N$ be the image of this embedding, $\hat T$ be the torsion subgroup of $\hat G$ and $\hat A\subseteq \hat G$ be the summand of $\hat G$ corresponding to $A\subseteq G$. So, the conclusions of Lemma~\ref{Nelem} must all hold.

We know that  $N$ is a summand of $T$ (we need to show it is a summand of $G$).  Since $N$ is elementary, there is a decomposition $N=N_A\oplus N'$ where $N_A=N\cap A$. Since $A$ is Bassian, for each prime $p$, $N_{A,p}$ is finite.

On $N'$, the projection $G=A\oplus E\to E$ is injective and its image is a summand of $E$. So there is a decomposition $G=A\oplus E'$, where $N'\subseteq E'$. Replacing $E$ by $E'$ and letting $N_E=N'$, we may assume there are decompositions $G=A\oplus E$ and $N=N_A\oplus N_E$, where $N_A\subseteq A$ and $N_E\subseteq E$.

There is a natural decomposition $G/N\cong (A/N_A)\oplus (E/N_E)$; let $\pi:G/N\to E/{N_E}$ be the obvious projection. Let $\mathcal P'$ be the set of primes such that $\pi(\hat A[p])\ne 0$. We claim that $\mathcal P'$ is actually finite. Otherwise, for each $p\in \mathcal P'$, let $x_p\in \hat A[p]$ with $\pi(x_p)\ne 0$.  If $M=\oplus_{p\in \mathcal P'}\langle x_p\rangle\subseteq \hat A$, then $\pi$ is an injection on $M$ and $\pi(M)$ is a summand of the elementary group $E/N_E$. Therefore, $M$ will be a summand of $G/N$, and hence of $\hat A$. This, however, contradicts that $A$ has no infinite elementary summands.

Let $\mathcal P$ be the cofinite collection of primes not in $\mathcal P'$. It follows that for all $p\in \mathcal P$ that $\hat A[p]\subseteq A/{N_A}$.  So for such $p\in \mathcal P$ we have
$$f_A^p(0)=f_{\hat A}^p(0)\leq f_{A/{N_A}}^p(0)=f_{A}^p(0)-r_p({N_A});$$
here the first equality is obvious, the next inequality is due to the fact that the torsion part of $\hat A$ is pure in $G/N$ and the last equality is due to the fact that $N_A$ is a summand of $T$. Consequently, for all $p\in \mathcal P$, $(N_A)_p=0$, so that ${N_A}$ is finite. Therefore, ${N_A}$ is a summand of $A$ and ${N_E}$ is a summand of $E$, so that $N$ is a summand of $G$, as required.
\end{proof}

We want to generalize Theorem~\ref{allfinite} by replacing the elementary group $E$ with any torsion group.  Here is the critical step in that program.

\begin{proposition}\label{frosty}
Suppose $A$ is a group that has no infinite elementary summands and $S$ is a torsion group with no elementary summands. Then $A\oplus S$ has no infinite elementary summands.
\end{proposition}

\begin{proof}
Suppose $E$ is an elementary summand of $A\oplus S$, say $A\oplus S=E\oplus H$; we must show that $E$ is finite.
Let $\pi: A\oplus S\to S$ be the obvious projection. For any prime $p$, we have $$S[p]=(pS)[p]\subseteq p(E_p\oplus H)\subseteq H.$$ Therefore, $\pi(E)\subseteq H$. It follows that $$E':=\{x-\pi(x):x\in E\}\cong E$$ is also a summand of $A\oplus S$. Since $E'\subseteq A$, we deduce that $E'$ is also a summand of $A$. So, by hypothesis, $E'$ must be finite, which implies that $E$ is finite, as required.
\end{proof}

The following is a slight generalization of Theorem~\ref{allfinite}.

\begin{corollary}\label{first}
Suppose $G$ is a {\rm{B+E}}-group with a decomposition $A\oplus S$, where $S$ is torsion. If $A$ has no infinite elementary summands, then $G$ is generalized-Bassian.
\end{corollary}

\begin{proof} As in Theorem~\ref{allfinite}, we can conclude that $A$ is Bassian.

Suppose $S=S_1\oplus E$, where $E$ is elementary and $S_1$ has no elementary summand. With Proposition~\ref{frosty} at hand, we have $A_1=A\oplus S_1$ also has no infinite elementary summand. So, by Theorem~\ref{allfinite}, the group $G=A_1\oplus E$ is generalized Bassian, as expected.
\end{proof}

\section{Tight subgroups of Bassian groups}\label{C}

The following illustrates how we can use tight subgroups to verify that certain groups satisfy the hypotheses of Theorem~\ref{allfinite}, so that Conjecture~\ref{conj} holds in this case.

\begin{proposition}
Suppose $A$ is a Bassian group with a tight subgroup $F$. Then any torsion summand of $A$ is finite.
\end{proposition}

\begin{proof}
Suppose $S$ is such a torsion summand of $A$; since each $S_p$ is finite, it will suffice to show that $S_p=0$ for almost all primes $p$.  Suppose $A=A'\oplus S$ and $F'=F\cap A'$; so clearly $F'$ is still free and $A/F'$ is torsion. Now, for almost all primes $p$, the localizations $F_p$ and $F_p'$ will necessarily coincide. So, for almost all primes $p$, we have $$f_{A'}^p(j)\geq f_{F_p'}^p(j)=f_{F_p}^p(j)=f_A^p(j)\geq f_{A'}^p(j),$$ for all $j<\omega$. Therefore, for almost all $p$, $f_{A}^p(j)=f_{A'}^p(j)$ for all $j$, so that $S_p=0$, as required.
\end{proof}

The same proof implies the following consequence.

\begin{corollary}\label{suffer}
Suppose $A$ is a Bassian group with a 0-tight subgroup $F$. Then any elementary summand of $A$ is finite.
\end{corollary}

\begin{corollary}\label{second}
If $G$ is a {\rm{B+E}}-group with a decomposition $A\oplus S$, where $S$ is torsion and $A$ has a 0-tight subgroup, then $G$ is generalized-Bassian.
\end{corollary}

\begin{proof} The existence of a 0-tight subgroup clearly implies that $f_A^p(0)$ is finite for all primes. And since each  $pA_p$ must clearly be finite, we can conclude that $A$ must be Bassian. By virtue of Corollary~\ref{suffer}, $A$ has no infinite elementary summands, so the result follows from Corollary~\ref{first}.
\end{proof}

We now observe that our conjecture holds for the class of Warfield groups, and in particular, the balanced-projective groups.

\begin{corollary}
Suppose $G$ is a Warfield group. Then $G$ is generalized Bassian if, and only if, it is a {\rm{B+E}}-group.
\end{corollary}

\begin{proof} Follows from a combination of Corollary~\ref{second} and Proposition~\ref{flag}.
\end{proof}

For groups of finite torsion-free rank and bounded $p$-torsion, the following statement describes when they have free quasi-essential subgroups that are {\it nice.}

\begin{proposition}\label{nice}
Suppose $G$ is a group with finite torsion-free rank and bounded $p$-torsion. Then the following are equivalent:

(a) Every free quasi-essential subgroup $F$ is nice in $G$;

(b) Some free quasi-essential subgroup $F$ is nice in $G$;

(c) For every prime $p$, the quotient $G_p/T_p$ is the direct sum of a free local module and a divisible module.
\end{proposition}

\begin{proof}
Suppose that $F$ is an arbitrary free quasi-essential subgroup of $G$. It will clearly suffice to show that $F$ is nice if, and only if, $G_p/T_p$, for all primes $p$,  is the direct sum of a free local module and a divisible module. And since $F$ is nice in $G$ if, and only if, $F_p$ is nice in $G_p$ for each prime, it suffice to assume that all groups mentioned are $p$-local.

Also, there is a splitting $G=T\oplus H$, where $H$ is torsion-free of finite rank. If $p^n T=0$, then since $F(n)\subseteq H$ and $F/F(n)$ is finite, $F$ is nice in $G$ if and only if $F(n)$ is nice in $H$. It follows that there is no loss of generality in assuming $H=G$ is torsion-free.

If $G=R\oplus D$, where $R$ is reduced and $D$ is divisible, then  $F(\infty)=F\cap D$ is nice in $D$, since every element of a corresponding coset has height $\infty$.  So $F$ is nice in $G$ if and only if $F_\infty:=F/F(\infty)$ is nice in $G/G(\infty)\cong R$. However, $F(\infty)$ is pure in $F$, and hence a summand, so that $F_\infty$ is free. It therefore follows that $F_\infty$ is nice in $R$ if, and only if, $R/F_\infty$ is finite, i.e., if, and only if, $R$ is free, as required.
\end{proof}

The following assertion allows us to replace a decomposition subgroup of a Warfield group with a general free nice quasi-essential subgroup $F\subseteq G$.

\begin{lemma}\label{tight}
Suppose $F$ is a free finite rank valuated group. Then $F$ has a nice quasi-essential realization $A$ that is a Bassian group if, and only if, for every prime $p$, there is an $n_p\in \omega$ such that $F_p(n_p)$ has a valuated decomposition $X_p\oplus F_p(\infty)$, where for every $x\in X_p\subseteq F_p(n_p)$, $\val x_p$ is $n_p$ plus the $p$-height of $x$ as an element of $X_p$. In fact, if $F$ satisfies this condition for each prime $p$, then we can construct $A$ containing $F$ as a 0-tight subgroup.
\end{lemma}

\begin{proof} If we are given the nice realization $A$, one argues as in Proposition~\ref{nice} that each $F_p$ has the required decomposition (or see the proof of Theorem~\ref{nice2}). So, assume we are given $F$ so that such decompositions always exist. There is clearly no loss of generality in assuming that for each prime $p$, $n_p\geq 2$.

For each such prime, let $b_{p,\infty,j}\in F$, for $j=1,\dots, d_{p,\infty}$, be a $\mathbb Z_{(p)}$-basis for the free module $F_p(\infty)$. Define a set of generators $y_{p,\infty, j, i}$ ($i<\omega$) subject to the relations $y_{p,\infty, j,0}=b_{p,\infty,j}$ and $py_{p,\infty, j, i+1}=y_{p,\infty, j,i }$.

Similarly, let $b_{p,n_p,j}\in F$, for $j=1,\dots, d_{p,n_p}$, be a $\mathbb Z_{(p)}$-basis for $X_p$. Define a set of generators $y_{p,n_p, j}$ subject to the relations $p^{n_p} y_{p,n_p, j}=b_{p,n_p,j}$.

Next, for $m=2, \dots n_p-1$, let $b_{p,m,j}\in F$, for $j=1,\dots, d_{p,m}$, project onto a basis for $F_p(m)/F_p(m+1)$. Define a set of generators $y_{p,m, j}$ subject to the relations $p^m y_{p,m, j}=b_{p,m,j}$.

Finally, let $b_{p,1,j}\in F$, for $j=1,\dots, d_{p,1}$, project onto a basis for $$F_p(1)/(pF_p+F_p(2)).$$ Define a set of generators $y_{p,1, j}$ subject to the relations $py_{p,1, j}=b_{p,1,j}$.

We let $A$ be the group generated by $F$ together with all the $y$s subject to the above relations. Note that for all primes $p$, $A_p$ is the $p$-local module generated by $F_p$ and all of the above generators of the form $y_{p, \dots}$. We want to show that $A$ is a realization of $F$, i.e., for all primes $p$ and all $\alpha\in \omega\cup \{\infty\}$, we have $F_p(\alpha)=p^\alpha A_p\cap F_p$.

If $U_p$ is the submodule of $A_p$ generated by $F_p(\infty)$ together with $y_{p,\infty,j,i}$ for all $j\leq d_{p,\infty}$, $i<\omega$, then it is easy to see that $U_p$ is divisible, torsion-free and $F_p(\infty)=F_p\cap U_p$. In fact, $U_p$ will be a divisible hull of $F_p(\infty)$ and hence a summand of any containing module.

Let $W\subseteq A$ be the subgroup generated by $F$ and all possible $y_{p,\infty, j, i}$. Setting $\val {y_{p,\infty,j,i}}_p=\infty$ and $\val {y_{p,\infty,j,i}}_q=\val {b_{p,\infty, j}}_q$ for all $j\leq d_{p,\infty}$, $i<\omega$ and $q\ne p$ extends the valuation on $F$ to one on $W$ with $W_p(\infty)=p^\infty W_p=U_p$. There is clearly a valuated decomposition $W_p(n_p)=X_p\oplus U_p$. It will suffice to show that $A$ is a realization of $W$; i.e., each $A_p$ is a realization of $W_p$.

\medskip

Note that for $m\leq n_p$, and $j\leq  d_{p,m}$,  if $p^s y_{p,m, j}\in W_p$, then $s\geq m$ and $\val {p^s y_{p,m, j}}_{W_p}\geq s$. It follows that $p^s A_p\cap W_p\subseteq W_p(s)$ for all $s<\omega$.

We  show by a reverse induction that $W_p(s)\subseteq p^s A_p\cap W_p$; i.e., $W_p(s)\subseteq p^s A_p$. We first verify this for $s=n_p$.
In the decomposition $W_p(n_p)=X_p\oplus U_p$, since the $b_{p,n_p, j}$ map on to a $\mathbb Z_{(p)}$-basis for $X_p$ it follows that
$$W_p(n_p)= \langle b_{p,n_p, j} \rangle+U_p= \langle p^{n_p}y_{p,n_p, j} \rangle+p^{n_p}U_p\subseteq p^{n_p} A.$$
So, if $s=n_p+k\geq n_p$, then
$
              W_p(s)=p^kW_p(n_p) \subseteq p^k p^{n_k} A=p^s A,
$
as required.

Suppose now that $1<s<n_p$ and this holds for $s+1$. Then
$$
                  W_p(s)=\langle b_{p,s,j}\rangle +W_p(s+1)\subseteq p^sA_p+p^{s+1}A_p=p^s A_p.
$$
Finally, if $s=1$ and it has been verified for $s=2$, then
$$
                  W_p(1)=\langle b_{p,1,j}\rangle +pW_p+W_p(2)\subseteq pA_p+pW_p+p^{2}A_p=pA_p.
$$

Therefore, $A$ is a realization of $W$ and $F$.

\medskip

Note that the quotient $A/W$ is a torsion group. If $T$ is the torsion subgroup of $A$ and $p$ is a prime, then $T_p\cap W_p=0$. So $T_p$ embeds in $A_p/W_p$ which is clearly finite. Therefore, each $T_p$ is finite, so that $A$ is Bassian.

\medskip

We next verify that $V$ is nice in $A$. To that end, let $p$ some prime. Clearly, $V_p$ is nice in $W_p$, since every coset has an element of infinite value. And $W_p$ is nice in $A_p$, since the corresponding quotient is finite. Therefore, each $V_p$ is nice in $A_p$, as required.

\medskip

We now claim that $F$ is $0$-tight in $A$. To that end, suppose $p$ is a prime and $x\in A$ represents an element of $U_0(A)$; i.e., $\val x_A=0$ and $\val {px}_A\geq 2$; we need to show that $x\in F+pA$.

Since $F$ and $F_p$, as well as $A$ and $A_p$ have isomorphic Ulm invariants, there is no loss of generality in assuming $F$ and $A$ are local.
So, we may assume that
$x$ is the sum of an element of $F$ plus terms of the form $sy_{p,m,j}$ and $ty_{p,\infty, j,i }$ where such $s, t$ are integers. Since each $y_{p,\infty, j,i }\in pA$, modulo $pA$, we may again ignore the last type of term. Similarly, whenever $m>1$, since $\val {px}_{A}\geq 2$, we can conclude $p$ divides $s$, so again modulo $pA$ we may ignore terms of the form $sy_{p,m,j}$ for $m>1$, as well. So we may assume $x$ is of form
$$
          x=   s_1 y_{p,1,1}+\cdots +s_{n_1} y_{p,1,n_1} +v \ \ \ (v\in F).
$$
Since
$$
           px = s_1 b_{p,1,1}+\cdots +s_{n_1} b_{p,1,d_{p,1}} +pv\in F
$$
has height at least $2$, $px-pv= s_1 b_{p,1,1}+\cdots s_{n_1} b_{p,1,d_{p,1}}$ represents the zero element of $F_p(1)/(pF_p+F_p(2))$.  Since these $b$s were chosen to project to a basis of this vector space, we can conclude that $p$ divides each $s$. Therefore, each $s_i y_{p,1,i}$ is in $\langle b_{p,1,i}\rangle \subseteq F$, so that $x\in F$, as required.
\end{proof}

The following result clearly generalizes the fact that any Warfield B+E-group is necessarily generalized Bassian. It allows us to replace a decomposition subgroup by any nice free quasi-essential subgroup.

\begin{theorem}\label{nice2}
Suppose $G$ is a group with a free nice quasi-essential subgroup $F$. Then $G$ is generalized Bassian if, and only if, it is a {\rm{B+E}}-group.
\end{theorem}

\begin{proof} One direction being immediate, suppose $G$ is a B+E-group. Note that by Proposition~\ref{quotient}, $G/F$ is totally projective, i.e., for each prime $(G/F)_p\cong G_p/F_p$ is the direct sum of a divisible and a bounded module.

We claim that $F$ (with the induced valuations) satisfies Lemma~\ref{tight}, so assume $p$ is some prime; we need to find $n_p$ and a decomposition $F_p=X_p\oplus F_p(\infty)$ as in that result.

Owing to Proposition~\ref{nice}, there is a decomposition $$G_p=T_p\oplus R_p\oplus D,$$ where $R_p$ is a free module and $D$ is divisible. If $p^{m} T_p=0$, then $F_p(m)\subseteq R_p\oplus D$. Find $n_p\geq m$ such that $X_p:=p^{n_p} R_p\subseteq F_p$. Since $$p^{n_p} (R_p\oplus D)=(p^{n_p} R_p) \oplus D,$$ we can conclude that  $F_p(n_p)=X_p\oplus F_p(\infty)$, and our decomposition behaves as required.

So, in view of Lemma~\ref{tight}, we can find a Bassian group $A$ that is a nice $0$-tight realization of $F$. Let $T^1$ be a second copy of the torsion subgroup of $A$ and $T^2$ be a second copy of the torsion subgroup of $G$. Since the Ulm invariants of $G\oplus T^1$ agree with those of $A\oplus T^2$, and the Ulm invariants of the valuated group $F$ are all finite, it follows that the relative Ulm invariants of $F$ in $G\oplus T^1$ agree with those of $F$ in $A\oplus T^2$. [Notice that, for a given prime $p$, both relative Ulm invariants will equal $f^p_{T^1\oplus T^2}(\alpha)-f^p_F(\alpha)$ when $\alpha$ is finite and $0$ otherwise.]

Therefore, the identity isometry $F\cong F$ extends to an isomorphism $G\oplus T^1\cong A\oplus T^2$.  It follows from Corollary~\ref{second} that $A\oplus T^2$ is generalized Bassian. Consequently, $G$, as a summand of a generalized Bassian group, will also be a generalized Bassian group, as required.
\end{proof}

As a valuable consequence, we derive the following.

\begin{corollary}\label{divisible}
Suppose $G$ is a group such that $G/T$ is divisible. Then $G$ is generalized Bassian if, and only if, it is a {\rm{B+E}}-group.
\end{corollary}

\begin{proof} Again, the necessity is clear. On the other hand, if $G$ is B+E, then according to Proposition~\ref{nice}, it has a free nice quasi-essential subgroup. So, the result follows at once from Theorem~\ref{nice2}, as desired.
\end{proof}

Of the groups mentioned in Section~\ref{A}, we have already noted that Conjecture~\ref{conj} holds for Warfield (and balanced-projective) group. It follows directly from Corollary~\ref{divisible} that it also holds for the other class mentioned there.

\begin{corollary} Suppose $G$ is a {\rm{PSP}}-group. Then $G$ is generalized Bassian if, and only if, it is a {\rm{B+E}}-group.
\end{corollary}

We have already mentioned that the class of B+E-groups is closed under arbitrary subgroups. The following makes this more explicit and demonstrates an even closer connection between those groups and the generalized Bassian groups.

\begin{corollary}\label{embeds}
A group $G$ is {\rm{B+E}} if, and only if, it embeds in a group that is generalized Bassian.
\end{corollary}

\begin{proof}
We know that any generalized Bassian group is B+E and that a subgroup of a B+E-group retains that property, so that sufficiency is clear.

Conversely, if $G$ is a B+E-group, and $D$ is a (torsion-free finite rank) divisible hull of $G/T$, then the surjectivity of the map $${\rm Ext}(D, T)\to {\rm Ext}(G/T,T)$$ implies that there is a group $G'$ containing $G$ (and hence $T$) such that $G'/T\cong D$ is divisible. Clearly, $G'$ will be a B+E-group. It follows from Corollary~\ref{divisible} that $G'$ is generalized Bassian, concluding the proof.
\end{proof}

Let us recall now that a group $G$ is {\it Hopfian} if, for any (possibly proper) subgroup $H$ of $G$, whenever $G\cong G/H$, it must be that $H=\{0\}$. It is principally known that all Bassian groups are Hopfian (see, e.g., \cite{CDG1}), which is {\it not} the case for generalized bassian groups, however (see, e.g., \cite{CDG2}). Nevertheless, the following necessary and sufficient condition is true, which proof is not so hard and follows from our results alluded to above, so we will omit the details in its proving.

\begin{proposition} A {\rm{B+E}}-group is Bassian if, and only if, it is Hopfian. In particular, the same holds for generalized Bassian groups.
\end{proposition}

\medskip
\medskip

\noindent {\bf Funding:} The work of the first-named author, P.V. Danchev, is partially supported by the Bulgarian National Science Fund under Grant KP-06 No. 32/1 of December 07, 2019 as well as by the Junta de Andaluc\'ia, Grant FQM 264, and by the BIDEB 2221 of T\"UB\'ITAK.

\end{document}